\documentclass[12pt]{article}
\usepackage[centertags]{amsmath}
\usepackage{amsfonts}
\usepackage{amssymb}
\usepackage{amsthm}
\usepackage{xcolor}
\input{amssym.def}
\usepackage{graphicx}
\usepackage{lineno}
\numberwithin{equation}{section}

\newtheorem{Definition}{Definition}[section]
\newtheorem{theorem}[Definition]{Theorem}
\newtheorem{lemma}[Definition]{Lemma}

\newtheorem{corollary}[Definition]{Corollary}
\newtheorem{example}[Definition]{Example}

\begin{document}

\title{\Large \bf  Interior ideal in regular and intra regular semigroup}
\author{Susmita Mallick   \\
\footnotesize{Department of Mathematics, Visva-Bharati
University,}\\
\footnotesize{Santiniketan, Bolpur - 731235, West Bengal, India}\\
\footnotesize{mallick.susmita11@gmail.com}}

\date{}
\maketitle

\begin{abstract}

Following G.Szasz \cite{SZASZ2} a subsemigroup $I$ of semigroup $S$
is called an interior ideal if $SIS\subseteq I$. In this paper we
explore the classes of regular semigroup and its different
subclasses by their interior ideals. Futhermore we introduce the
strongly prime, prime, semiprime, strongly irreducible and
irreducible interior ideals of semigroups and also characterize
those semigroups for which each interior ideal is strongly prime.
Some important interplay between the classes of all interior ideals
and other ideals are given here. In addition to this we present
different characterizations of semigroups by their minimal interior
ideals.

\end{abstract}

\section{Introduction and preliminaries}

Like ring theory, ideals play a significant role in the theory of
semigroups. The classes of ideals in semigroup theory is not just a
generalization of ideals in ring theory. In some cases ideals in
semigroups failed to play the similar role of the ideals in ring
theory. These obstacles have been forcing the semigroup theorists to
study different kind of ideals to characterize semigroups. Interior
ideal is one of them. The concept of interior ideal in a semigroup
has been introduced by S. Lajos \cite{LAJOS1}. G.
Szasz\cite{SZASZ1},\cite{SZASZ2} has a significant contribution in
the study of interior ideals in semigroups. This work is highly
motivated by the works of G.Szasz \cite{SZASZ1},\cite{SZASZ2}.

Let $(S,\cdot)$ be a semigroup. By a subsemigroup of $S$ we mean a
non-empty subset $A$ of $S$ such that $AA\subseteq A$. A non-empty
subset $A$ of $S$ is left(resp. right) ideal of $S$ if $SA\subseteq
A$ (resp. $AS\subseteq A$). An ideal of $S$ is a non-empty subset of
$S$ which is both left as well as right ideal. Following Lajos
\cite{S. Lajos 1}, a subsemigroup $B$ of $S$ is a bi-ideal of $S$ if
$BSB\subseteq B$ and a subsemigroup $Q$ of $S$ is a quasi-ideal if
$QS\cap SQ\subseteq Q$. A subsemigroup  $I$ of a semigroup $S$ is
called an interior ideal \cite{SZASZ2} of $S$ if $SIS\subseteq I$. A
regular semigroup \cite{Howie} is a semigroup $S$ in which every
element is regular that is for each element $a\in S$ there exist an
element $x\in S$ such that $a=axa$ that is $a\in aSa$. $S$ is said
to be intra-regular if for each element $a\in S$, $a\in Sa^{2}S$.
$S$ is called duo semigroup \cite{Lajos3} if every one-sided ideal
of $S$ is two sided. For any element $a$ in a semigroup $S$, the
smallest left ideal, right ideal, ideal of $S$ containing $a$ is
$L(a)=Sa\cup {a}$,$R(a)=aS\cup {a}$, $I(a)= {a}\cup Sa\cup aS\cup
SaS$ called respectively the principal left ideal, principal right
ideal, principal ideal of $S$ generated by $a$. In \cite{SZASZ2} G.
Szasz defined principal interior ideal of $S$ generated by an
element $a$ in $S$ denoted by IN$(a)$ where IN$(a)={a}\cup a^{2}\cup
SaS$. For any fundamental query of semigroup theory one can go
through the following book \cite{Howie}.

\begin{lemma}\label{ideal}
Every ideal is an interior ideal of a semigroup.
\end{lemma}

But the converse is not true.
\begin{example}
Consider the following semigroup $S=\{a,b,c,d\}$. Define a binary
operation '$\cdot$' on $S$ as follows:

\begin{tabular}{|c|c|c|c|c|}
  \hline

  . & a & b & c & d  \\ \hline
  a & a & c & b & d  \\ \hline
  b & c & d & d & d  \\ \hline
  c & b & d & d & d  \\ \hline
  d & d & d & d & d  \\

  \hline

\end{tabular}

Here $\{b,d\}$ and $\{c,d\}$ are interior ideals of $S$ but none of
them neither left ideal nor right ideal of $S$.

\end{example}

This makes the study of interior ideal interesting despite of a
revolutionary study of the theory of ideals in last five decades.
The notion of interior ideals is not new, it is significant
contribution which was initiated by various authors, specially G.
Szasz \cite{SZASZ1}, \cite{SZASZ2}. But the works of G. Szasz
\cite{SZASZ1}, \cite{SZASZ2} made us to find a different way to
study interior-ideals again. Our works enable one to see some
important characterizations of semigroups by their interior ideals.

\begin{Definition}
 An  interior ideal $I$ of $S$ is semiprime if for any interior
ideal $A$ of $S$, $A^{2}\subseteq I$ implies $A\subseteq I$. Whereas
An interior ideal $I$ of a semigroup $S$ is completely semiprime if
for any $a\in S$, $a^{2}\in I$ implies $a\in I$. An interior ideal
$I$ of a semigroup $S$ is called a prime(strongly prime) interior
ideal of $S$ if $I_{1}I_{2}\subseteq I(I_{1}I_{2}\cap
I_{2}I_{1}\subseteq I)$ implies $I_{1}\subseteq I$ or
$I_{2}\subseteq I$ for any two interior ideals $I_{1}, I_{2}$ of
$S$.

\end{Definition}

\begin{Definition}
An interior ideal $I$ of a semigroup $S$ is called an irreducible
(strongly irreducible) interior ideal if for any two interior ideals
$I_{1}, I_{2}$ of $S$, $I_{1}\cap I_{2}= I$($I_{1}\cap
I_{2}\subseteq I$) implies $I_{1}=I$ or $I_{2}=I$($I_{1}\subseteq I$
or $I_{2}\subseteq I$).

\end{Definition}

Every strongly prime interior ideal is a prime interior ideal of
$S$. Every  prime interior ideal is a semiprime interior ideal of
$S$ but a prime interior ideal is not necessarily strongly prime and
a semiprime interior ideal is not necessarily prime.

\begin{example}
Consider the following semigroup $S=\{a,b,c\}$. Define a binary
operation '$\cdot$' on $S$ as follows:

\begin{tabular}{|c|c|c|c|}
  \hline

  . & a & b & c \\ \hline
  a & a & a & a \\ \hline
  b & a & b & a \\ \hline
  c & a & a & c \\
  \hline
\end{tabular}

Interior ideals of $S$ are $\{a\}, \{a,b\}, \{a,c\}$ and
$\{a,b,c\}$.

Here $\{a,b,c\}$ is strongly prime interior ideal hence prime and
semiprime ideal but $\{a\}$ is not strongly prime since
$\{a,b\}\{a,c\}\cap \{a,c\}\{a,b\}\subseteq \{a\}$ but neither
$\{a,b\}\subseteq \{a\}$ nor $\{a,c\}\subseteq \{a\}$.
\end{example}

\begin{lemma}\label{RL}\cite{Lajos2}
Let $S$ be a semigroup. $S$ is regular if and only if $R\cap L=RL$
for every right ideal $R$ and left ideal $L$ of $S$.
\end{lemma}

\begin{lemma}\label{quasi-ideal}\cite{S. Lajos}
Let $S$ be a semigroup and $Q$ be a quasi-ideal of $S$. $S$ is
regular if and only if $QSQ=Q$.
\end{lemma}

\begin{lemma}\label{bi-ideal}\cite{S. Lajos}
A semigroup $S$ is regular if and only if  $BSB=B$ for each bi-ideal
$B$ of $S$.
\end{lemma}

\begin{lemma}\label{left-quasi}\cite{MOIN}
Let $S$ be a semigroup. Every right(left) ideal of $S$ is a
quasi-ideal of $S$.
\end{lemma}

\begin{theorem}\cite{S. Lajos}\label{bi-quasi}
Let $S$ be a regular semigroup. Then every bi-ideal of $S$ is a
quasi-ideal, and conversely.

\end{theorem}

\begin{theorem}\label{duo}\cite{S. Lajos 1}
Let $S$ be a semigroup. If $S$ is regular duo then every bi-ideal of
$S$ is an ideal of $S$.
\end{theorem}

\section{Interior ideals in regular and intra-regular semigroup}

\begin{theorem}
Arbitrary intersection of interior ideals (if non-empty) of a
semigroup $S$ is an interior ideal of $S$.
\end{theorem}
\begin{proof}
Let $\{A_{i}\}_{i\in \Delta}$ ($\Delta$ denotes any indexing set) be
a family of interior ideals of a semigroup $S$ and $T=\cap_{i\in
\Delta}A_{i}$ is non empty and thus a subsemigroup of $S$. Now for
$i\in \Delta$, $SA_{i}S\subseteq A_{i}$. We have to show that
$STS\subseteq T$. Let $x\in STS$ implies $x=s_{1}ys_{2}$ where
$s_{1},s_{2}\in S$ and $y\in T=\cap_{i\in \Delta}A_{i}\subseteq
A_{i}$, for all $i\in \Delta$ which implies that $x=s_{1}ys_{2}\in
SA_{i}S$ for all $i\in \Delta$. Since $SA_{i}S\subseteq A_{i}$, for
all $i\in \Delta$ we have $x\in \cap_{i\in \Delta}A_{i}=T$ thus
$STS\subseteq T$. Hence $T=\cap_{i\in \Delta}A_{i}$ is an interior
ideal of $S$.
\end{proof}

\begin{theorem}
If $I$ is an interior ideal and $T$ is a subsemigroup of a semigroup
$S$ then $I\cap T$ is an interior ideal of $T$, provided $I\cap T$
is non-empty.
\end{theorem}
\begin{proof}
Let $I$ be an interior ideal of $S$ and $T$ be a subsemigroup of
$S$. Let $I\cap T$ be non-empty and thus a subsemigroup. Now
$T(I\cap T)T\subseteq TIT\cap TTT\subseteq SIS\cap TTT\subseteq
I\cap T$. Hence $I\cap T$ is an interior ideal of $T$.
\end{proof}

\begin{theorem}\label{regular}
Let $S$ be a semigroup. If $S$ is regular then $I=SIS$, for every
interior ideal $I$ of $S$.
\end{theorem}
\begin{proof}
Let $S$ be a regular semigroup and $I$ be an interior ideal of $S$.
Take any $a\in I$. Therefore $a=axa$ for some $x\in S$. Thus $a\in
aSa\subseteq aSaSa\subseteq SIS$ which implies $I\subseteq SIS$.
Hence $I=SIS$.

\end{proof}

\begin{theorem}
Following conditions are equivalent on a semigroup $S$:

\begin{enumerate}
\item \vspace{-.4cm}
$S$ is regular.

\item \vspace{-.4cm}
For a quasi-ideal $Q$ and an ideal $J$ of $S$, $Q\cap J=QJQ$

\item \vspace{-.4cm}
For a quasi-ideal $Q$ and an interior ideal $I$ of $S$, $Q\cap
I=QIQ$

\item \vspace{-.4cm}
For an interior ideal $I$ and a bi-ideal $B$ of $S$, $I\cap
B=BIB$

\end{enumerate}
\end{theorem}

\begin{proof}

$(1)\Rightarrow(2)$: Let $Q$ be a quasi-ideal and $J$ be an ideal of
$S$. Now $QJQ\subseteq QSQ\subseteq QS\cap SQ\subseteq Q$. Again
$QJQ\subseteq SJS$. Since $S$ is regular, by Theorem \ref{regular}
and Lemma \ref{ideal} we have $QJQ\subseteq SJS=J$. Hence
$QJQ\subseteq Q\cap J$. Also let $z\in Q\cap J$, $z\in Q$ and $z\in
J$. Since $S$ is regular, by Lemma \ref{regular} and Lemma
\ref{quasi-ideal} we have $z\in zSz\subseteq (zSz)S(zSz)\subseteq
(QSQ)(SJS)Q\subseteq QJQ$ thus $Q\cap J\subseteq QJQ $. Hence $Q\cap
J=QJQ$.

$(2)\Rightarrow(3)$: This implication follows from Lemma
\ref{ideal}.

$(3)\Rightarrow(1)$: Let $Q$ be any quasi-ideal of $S$. Therefore
$QSQ=Q\cap S=Q$. Hence by Lemma \ref{quasi-ideal} $S$ is regular.

$(1)\Rightarrow(4)$: Let $B$ be a bi-ideal and $I$ be an interior
ideal of $S$. Now $BIB\subseteq BSB\subseteq B$. Again $BIB\subseteq
SIS\subseteq I$. Hence $BIB\subseteq B\cap I$. Let $a\in B\cap I$.
Now $a=axa\subseteq (aSa)(SaS)a\subseteq (BSB)(SIS)B\subseteq BIB$.
Thus $B\cap I=BIB$.

$(4)\Rightarrow(1)$: Take $B$ be any bi-ideal of $S$. Therefore
$BSB=B\cap S=B$. Hence by Lemma \ref{bi-ideal} $S$ is regular.

\end{proof}

\begin{theorem}
Following conditions are equivalent on a semigroup $S$:

\begin{enumerate}
\item \vspace{-.4cm}
$S$ is regular.
\item \vspace{-.4cm}
$B\cap I\cap L\subseteq BIL$, for a bi-ideal $B$, left ideal $L$ and
interior ideal $I$ of $S$.
\item \vspace{-.4cm}
$Q\cap I\cap L\subseteq QIL$, for a quasi ideal $Q$, left ideal $L$
and interior ideal $I$ of $S$.
\item \vspace{-.4cm}
$B\cap I\cap R\subseteq RIB$, for a bi-ideal $B$, right ideal $R$
and interior ideal $I$ of $S$.
\item \vspace{-.4cm}

$Q\cap I\cap R\subseteq RIQ$, for a quasi ideal $Q$, right ideal $R$
and interior ideal $I$ of $S$.

\end{enumerate}
\end{theorem}
\begin{proof}
Here we proof the implications $(1)\Rightarrow(2)\Rightarrow
(3)\Rightarrow (1)$ and $(1)\Rightarrow(4)\Rightarrow (5)\Rightarrow
(1)$.

$(1)\Rightarrow(2)$: Let $a\in B\cap I\cap L$. Then $a\in aSa$. Now
$aSa\subseteq (aSa)(SaS)a\subseteq BIL$.

$(2)\Rightarrow(3)$: This is obvious.

$(3)\Rightarrow(1)$: Consider a right ideal $R$ and a left ideal $L$
of $S$. Then by $(3)$ we have $R\cap S\cap L\subseteq RSL\subseteq
RL$ which gives $R\cap L\subseteq RL$ and $RL\subseteq R\cap L$.
Thus we get $R\cap L=RL$. Hence by Lemma \ref{RL} $S$ is regular.

$(1)\Rightarrow(4)$: Let $a\in B\cap I\cap R$. Hence $a\in
aSa\subseteq a(SaS)(aSa)\subseteq RIB$. Hence $B\cap I\cap
R\subseteq RIB$.

$(4)\Rightarrow(5)$: This is obvious.

$(5)\Rightarrow(1)$:Consider a right ideal $R$ and left ideal $L$,
by (5), we have $L\cap S\cap R\subseteq RSL\subseteq RL$ implies
$R\cap L\subseteq RL$. Again $RL\subseteq R\cap L$ always holds.
Hence $R\cap L=RL$. Therefore $S$ is regular, by Lemma \ref{RL}.

\end{proof}

\begin{theorem}\label{interior-regular}
Let $S$ be a regular semigroup then a non-empty subset of $S$ is an
ideal if and only if it is an interior ideal.
\end{theorem}
\begin{proof}
First assume that $S$ is a regular semigroup. A non empty subset $I$
of $S$ is an ideal of $S$, then $I$ is interior ideal, by Lemma
\ref{ideal}.

Conversely, suppose that a non-empty subset $I$ of $S$ is an
interior ideal of $S$. Since $S$ is regular, by Theorem
\ref{regular} $I=SIS$. Therefore $SI=S(SIS)\subseteq SIS\subseteq
I$. Similarly $IS\subseteq I$. Hence $I$ is an ideal of $S$.

\end{proof}
\begin{corollary}
In an intra-regular semigroup, an ideal and an interior ideal
coincide.
\end{corollary}

\begin{theorem}
In an intra-regular semigroup a proper interior ideal is semiprime.

\end{theorem}
\begin{proof}
Let $S$ be an intra-regular semigroup and $P$ be a proper interior
ideal of $S$. Take any interior ideal $A$ of $S$ such that
$A^{2}\subseteq P$. For any $a\in A$, $a\in Sa^{2}S\subseteq
SPS\subseteq P$. Hence $A\subseteq P$. Therefore $P$ is a semiprime
interior ideal of $S$.

\end{proof}

\begin{theorem}
Let $S$ be a semigroup. Then $S$ is intra-regular if and only if
each interior ideal of $S$ is completely semiprime.

\end{theorem}
\begin{proof}
First suppose that $S$ is intra-regular. Let $P$ be a proper
interior ideal of $S$. Let for any element $a\in S$ such that
$a^{2}\in P$. Now we have $a\in Sa^{2}S\subseteq SPS\subseteq P$
implies $a\in P$. Hence $P$ is a completely semiprime interior ideal
of $S$.

Conversely, assume that each interior ideal of $S$ is completely
semiprime. Take any $a\in S$. Let $I=Sa^{2}S$ then $SIS\subseteq I$.
Thus $I$ is an interior ideal of $S$. By assumption, $I$ is
completely semiprime. Now we have $a(a^{2})a\in Sa^{2}S=I$ so that
$(a^{2})(a^{2})\in I$ which implies that $a^{2}\in I$. Since $I$ is
completely semiprime $a\in I$, that is, $a\in Sa^{2}S$. Hence $S$ is
intra-regular.

\end{proof}

\begin{theorem}
Let $S$ be a semigroup.

\begin{enumerate}
\item \vspace{-.4cm}

If $S$ is regular duo then every bi-ideal of $S$ is an interior
ideal of $S$.

\item \vspace{-.4cm}
If $S$ is regular duo then every quasi-ideal of $S$ is an interior
ideal of $S$.
\end{enumerate}
\end{theorem}

\begin{proof}

$(1)$: It follows directly from Theorem \ref{duo} and Lemma
\ref{ideal}.

$(2)$: Since $S$ is a regular semigroup. Then following
\ref{bi-quasi} every quasi-ideal of $S$ is a bi-ideal and hence from
previous result we conclude that every quasi-ideal of $S$ is an
interior ideal.

\end{proof}

\begin{Definition}
A semigroup $S$ is said to be interior-simple semigroup if $S$ has
no non trivial interior ideals other than $S$ itself.
\end{Definition}

\begin{theorem}

In a semigroup $S$, following statements are equivalent:
\begin{enumerate}
\item \vspace{-.4cm}
$S$ is interior-simple semigroup.
\item \vspace{-.4cm}
$SaS=S$, for all $(0\neq)a\in S$.
\item \vspace{-.4cm}
IN$(a)=S$, for all $(0\neq)a\in S$.

\end{enumerate}
\end{theorem}
\begin{proof}

$(1)\Rightarrow(2)$: Suppose that $S$ is an interior-simple
semigroup. For any $(0\neq)a\in S$ it is evident that $SaS$ is an
interior ideal of $S$. Hence $SaS=S$.

$(2)\Rightarrow(1)$: Suppose that $SaS=S$, for any $(0\neq)a\in S$.
Let $I$ be an interior ideal of $S$. For any $(0\neq)b\in I$,
$SbS=S$, by $(2)$. Hence $SbS\subseteq SIS\subseteq I$ so that
$S\subseteq I$ Thus $S=I$. Hence $S$ is an interior simple
semigroup.

$(1)\Rightarrow(3)$: Let $S$ be an interior-simple semigroup. For
any $(0\neq)a\in S$, IN$(a)= (a\cup a^{2}\cup SaS)$ but $SaS=S$.
Hence IN$(a)= (a\cup a^{2}\cup SaS)=a\cup a^{2}\cup S=S$.

$(3)\Rightarrow(1)$: Let $I$ be an interior ideal of $S$ then for
any $(0\neq)a\in S$, IN$(a)=S$, by $(3)$. Hence $S=$IN$(a)\subseteq
I$. Therefore $I=S$. Hence $S$ is interior simple.

\end{proof}

\begin{theorem}
Every strongly irreducible semiprime interior ideal of a semigroup
$S$ is a strongly prime interior ideal.
\end{theorem}
\begin{proof}
Let $I$ be a strongly irreducible semiprime interior ideal of a
semigroup $S$ such that $I_{1}I_{2}\cap I_{2}I_{1}\subseteq I$ for
any two interior ideals $I_{1},I_{2}$ of $S$. Now we have
$(I_{1}\cap I_{2})^{2}\subseteq I_{1}I_{2}\cap I_{2}I_{1}\subseteq
I$. Since $I$ is semiprime $(I_{1}\cap I_{2})^{2}\subseteq I$ gives
that $I_{1}\cap I_{2}\subseteq I$. This together with the strongly
irreducibility of $I$ yields that either $I_{1}\subseteq I$ or
$I_{2}\subseteq I$. So $I$ is strongly prime interior ideal.
\end{proof}

\begin{theorem}\label{z}
Let $I$ be an interior ideal of a semigroup $S$ and $a\in S$ such
that $a\not\in I$ then there exists an irreducible interior ideal
$B$ of $S$ such that $I\subseteq B$ and $a\not\in B$.
\end{theorem}
\begin{proof}
Let $\mathcal{A}$ be the collection of all interior ideal of $S$
which contain $I$ but does not contain $a$. Hence $\mathcal{A}$ is
non empty, because $I\in \mathcal{A}$. The collection $\mathcal{A}$
is partially ordered set under inclusion. If $\mathcal{C}$ is any
totally ordered subset of $\mathcal{A}$ then $\cup_{J\in
\mathcal{C}}J$ is interior ideal of $S$ containing $I$. Hence by
Zorn's lemma there exists a maximal element $B$ in $\mathcal{A}$. We
show that $B$ is an irreducible interior ideal. Let $C$ and $D$ be
two interior ideals of $S$ such that $B=C\cap D$. If both $C$ and
$D$ properly contain $B$ then $a\in C$ and $a\in D$. Hence $a\in
C\cap D=B$ this contradicts the fact that $a\not\in B$ thus $B=C$
and so $B=D$.
\end{proof}

\begin{theorem}
For an regular semigroup $S$, the following assertations are
equivalent:

\begin{enumerate}
\item \vspace{-.4cm}

$I^{2}=I$ for every interior ideal $I$ of $S$.

\item \vspace{-.4cm}
$I_{1}\cap I_{2}=I_{1}I_{2}\cap I_{2}I_{1}$ for all interior ideals
$I_{1}$ and $I_{2}$ of $S$.

\item \vspace{-.4cm}
Each interior ideal of $S$ is semiprime.

\item \vspace{-.4cm}
Each proper interior ideal of $S$ is the intersection of irreducible
semiprime interior ideals of $S$ which contains it.

\end{enumerate}
\end{theorem}
\begin{proof}
$(1)\Rightarrow(2)$: Let $I_{1}$ and $I_{2}$ be any two interior
ideals of semigroup $S$, then by hypothesis, $I_{1}\cap
I_{2}=(I_{1}\cap I_{2})(I_{1}\cap I_{2})\subseteq I_{1}I_{2}$.
Similarly, $I_{1}\cap I_{2}\subseteq I_{2}I_{1}$. Hence $I_{1}\cap
I_{2}\subseteq I_{1}I_{2}\cap
I_{2}I_{1}$$\cdot\cdot\cdot\cdot\cdot\cdot\cdot\cdot\cdot\cdot$$(1)$

Now $I_{1}I_{2}$ and $I_{2}I_{1}$ are interior ideals being product
of interior ideals in a regular semigroup. Also $I_{1}I_{2}\cap
I_{2}I_{1}$ is an interior ideal. Thus, $I_{1}I_{2}\cap
I_{2}I_{1}=(I_{1}I_{2}\cap I_{2}I_{1})^{2}=(I_{1}I_{2}\cap
I_{2}I_{1})(I_{1}I_{2}\cap I_{2}I_{1})\subseteq
I_{1}I_{2}I_{2}I_{1}\subseteq SI_{2}S\subseteq I_{2}$. Similarly,
$I_{1}I_{2}\cap I_{2}I_{1}\subseteq I_{1}$. Hence $I_{1}I_{2}\cap
I_{2}I_{1}\subseteq I_{1}\cap
I_{2}$$\cdot\cdot\cdot\cdot\cdot\cdot\cdot\cdot\cdot\cdot$$(2)$.

Therefore from $(1)$ and $(2)$, $I_{1}\cap I_{2}=I_{1}I_{2}\cap
I_{2}I_{1}$ for all interior ideals $I_{1}$ and $I_{2}$ of $S$.

$(2)\Rightarrow (3)$: Let $I_{1}$ and $I$ be interior ideals of $S$
such that $I_{1}^{2}\subseteq I$. By our hypothesis, $I_{1}=
I_{1}\cap I_{1}=I_{1}I_{1}\cap I_{1}I_{1}=I_{1}^{2}$ implies
$I_{1}\subseteq I$. Hence every interior ideal of S is semiprime.

$(3)\Rightarrow (4)$: Let $I$ be a proper interior ideal of $S$.
Then $I$ is contained in the intersection of all irreducible
interior ideals $\{I_{\alpha}: \alpha\in \Delta\}$ of $S$ which
contain $I$ that is $I\subseteq \cap_{\alpha\in \Delta}I_{\alpha}$.
Assume that $\cap_{\alpha\in \Delta}I_{\alpha}\nsubseteq I$ then
there exists $a\in \cap_{\alpha\in \Delta}I_{\alpha}$ such that
$a\not\in I$. Hence by Theorem \ref{z} yields that there exists an
irreducible interior ideal $J$ of $S$ which contains $I$ but does
not contain $a$ which implies that $a\not\in \cap_{\alpha\in
\Delta}I_{\alpha}$, a contradiction. Hence $I$ is the intersection
of all irreducible interior ideals of $S$ which contains it. By our
assumption, every interior ideal is semiprime and so each interior
ideal is the intersection of irreducible semiprime interior ideals
of $S$ containing it.

$(4)\Rightarrow (1)$: Let $I$  be an interior ideal of $S$. If
$I^{2}=S$ then clearly I is idempotent that is $I^{2}=I$. If
$I^{2}\neq S$, then $I^{2}$ is a proper interior ideal of $S$
containing $I^{2}$ and so by hypothesis $I^{2}=\cap_{\alpha} \{
I_{\alpha}: I_{\alpha}$ is irreducible semiprime interior ideals of
$S$\}. Since each $I_{\alpha}$ is irreducible semiprime interior
ideal, $I\subseteq I_{\alpha}$, $\forall \alpha$ and so $I\subseteq
\cap_{\alpha} I_{\alpha}=I^{2}$. Hence each interior ideal in $S$ is
idempotent.

\end{proof}

\begin{theorem}
In a semigroup $S$ the following assertions are equivalent:
\begin{enumerate}
\item \vspace{-.4cm}
The set of interior ideals of $S$ is totally ordered under
inclusion.

\item \vspace{-.4cm}
Each interior ideal of $S$ is strongly irreducible.

\item \vspace{-.4cm}
Each interior ideal of $S$ is irreducible.

\end{enumerate}

\end{theorem}
\begin{proof}
$(1)\Rightarrow(2)$:First we assume that The set of interior ideals
of $S$ is totally ordered under inclusion.  Let $I$ be an arbitrary
interior ideal of $S$ and $I_{1},I_{2}$ two interior ideals of $S$
such that $I_{1}\cap I_{2}\subseteq I$. Since the set of all
interior ideals of $S$ is totally ordered, either $I_{1}\subseteq
I_{2}$ or $I_{2}\subseteq I_{1}$. Thus either $I_{1}\cap
I_{2}=I_{1}$ or $I_{2}\cap I_{1}=I_{2}$. Hence either
$I_{1}\subseteq I$ or $I_{2}\subseteq I$, thus $I$ is strongly
irreducible.

$(2)\Rightarrow (3)$: Let $I$ be an arbitrary interior ideal of $S$
and $I_{1},I_{2}$ two interior ideals of $S$ such that $I_{1}\cap
I_{2}=I$. Then $I\subseteq I_{1}$ or $I\subseteq I_{2}$. By
hypothesis, either $I_{1}\subseteq I$ or $I_{2}\subseteq I$. Hence
either $I_{1}=I$ or $I_{2}=I$ that is $I$ irreducible interior
ideal.

$(3)\Rightarrow (1)$: Let $I_{1}$ and $I_{2}$ be any two interior
ideals of $S$. Then $I_{1}\cap I_{2}$ is a interior ideal of $S$.
Also $I_{1}\cap I_{2}=I_{1}\cap I_{2} $. So by our hypothesis
$I_{1}=I_{1}\cap I_{2}$ or $I_{2}=I_{1}\cap I_{2}$, that is either
$I_{1}\subseteq I_{2}$ or $I_{2}\subseteq I_{1}$. Hence the set of
all interior ideals of $S$ is totally ordered.

\end{proof}

We now study minimality of interior ideals.

\begin{Definition}
An interior ideal $I$ of a semigroup $S$ is said to be a minimal
interior ideal of $S$ if $I$ does not contain any other proper
non-zero interior ideal of $S$.
\end{Definition}

\begin{theorem}\label{minimal}
 Let $S$ be a semigroup and $I$ be an interior ideal of $S$ then
following statements are equivalent:
\begin{enumerate}
\item \vspace{-.4cm}
$I$ is a minimal interior ideal of $S$.
\item \vspace{-.4cm}
$I=SaS$, for all $(0\neq)a\in S$.
\item \vspace{-.4cm}
$I=$IN$(a)$, for all $(0\neq)a\in S$.

\end{enumerate}

\end{theorem}
\begin{proof}
Let $S$ be a semigroup and $I$ be an interior ideal of $S$.

$(1)\Rightarrow(2)$: Let $(0\neq)a\in I$. Therefore $SaS\subseteq
SIS\subseteq I$ and $SaS$ is an interior ideal of $S$. Then by
minimality of $I$ we conclude that $I=SaS$.

$(2)\Rightarrow(1)$: Suppose that $J$ be any interior ideal of $S$
contained in $I$. For any $(0\neq)a\in J\subseteq I=SaS$. Also
$I=SaS\subseteq SJS\subseteq J$. Therefore $I=J$. Hence $I$ is a
minimal interior ideal of $S$.

$(1)\Rightarrow(3)$: Take any $(0\neq)a\in I$. Now IN$(a)=(a\cup
a^{2}\cup SaS)\subseteq I$. But $I$ is minimal so $I$=IN$(a)$ for
any $(0\neq)a\in I$.

$(3)\Rightarrow(1)$: Let $J$ be any interior ideal contained in $I$.
Now for any $(0\neq)a\in J\subseteq I$, $I=$IN$(a)\subseteq J$. So
$I=J$ thus $I$ is a minimal interior ideal of $S$.

\end{proof}

\begin{theorem}
A proper interior ideal of a semigroup $S$ is minimal if and only if
the intersection of any two distinct proper interior ideal of $S$ is
empty.
\end{theorem}

\begin{proof}
First assume that any proper interior ideal of a semigroup $S$ is
minimal. Let $A$ and $B$ be any two proper distinct interior ideal
of $S$ such that $A\cap B\neq \varnothing$ then  $S(A\cap
B)S\subseteq SAS\cap SBS\subseteq A\cap B$ that is $A\cap B$ is an
interior ideal of $S$. By hypothesis $A$ and $B$ are minimal
interior ideals of $S$ which implies that $A=A\cap B=B$ so $A=B$, a
contradiction. Hence $A\cap B=\varnothing$.

Conversely, suppose that the intersection of any two distinct proper
interior ideal is empty. Hence any proper interior ideal of $S$ does
not contain any proper interior ideal of $S$. Hence each proper
interior ideal of $S$ is a minimal interior ideal of $S$.
\end{proof}
\begin{theorem}\label{i(a)=i(b)}
Let $S$ be a semigroup and $I$ be an interior ideal of $S$. Then $I$
is a minimal if and only if IN$(a)$=IN$(b)$, for all
$(0\neq)a,(0\neq)b\in I$.

\end{theorem}
\begin{proof}
Assume that $I$ is a minimal interior ideal of $S$. Take any
$(0\neq)a,(0\neq)b\in I$. Then $I$=IN$(a)$, $I$=IN$(b)$, by Theorem
\ref{minimal}. Therefore IN$(a)$=IN$(b)$, for all
$(0\neq)a,(0\neq)b\in I$.

Conversely assume that IN$(a)$=IN$(b)$, for every
$(0\neq)a,(0\neq)b\in I$. Let $J$ be any interior ideal of $S$ such
that $J\subseteq I$. Let $(0\neq)x\in J$. Now for any $(0\neq)y\in I
$, we have IN$(x)$=IN$(y)$, by given condition. Since $y\in $
IN$(y)$. We have $y\in $ IN$(x)\subseteq J$. Therefore $I\subseteq
J$. Thus $I=J$ which concludes $I$ is a minimal interior ideal of
$S$.

\end{proof}

The Green's relation\cite{Howie} $\mathcal{L}, \mathcal{R},
\mathcal{J}, \mathcal{H}$ on a semigroup $S$ are defined as follows:
\begin{center}
$a\mathcal{L} b$ if and only if $L(a)=L(b)$.\\
$a\mathcal{R} b$ if and only if $R(a)=R(b)$.\\
$a\mathcal{J} b$ if and only if $I(a)=I(b)$.\\
$\mathcal{H}=\mathcal{L}\cap \mathcal{R}$
\end{center}

A relation $\mathcal{I}$ on $S$ is defined as for any $a,b\in S$,
$a\mathcal{I} b$ if and only if IN$(a)$=IN$(b)$. It is thus evident
that $\mathcal{J}\subseteq  \mathcal{I}$.

\begin{theorem}\label{I CLASS}
If $I$ is an interior ideal of a semigroup $S$, then $I$ is minimal
interior ideal if and only if $I$ is an $\mathcal{I}$-class.

\end{theorem}
\begin{proof}
Let $I$ be an interior ideal of $S$. Assume that $I$ is a minimal
interior ideal of $S$. Take any $(0\neq)a,(0\neq)b\in I$. Thus by
Theorem \ref{minimal}, $I$=IN$(a)$ and $I$=IN$(b)$. Hence
IN$(a)$=IN$(b)$ implies $a\mathcal{I} b$. Thus $I$ is an
$\mathcal{I}$-class.

Conversely assume that $I$ is an $\mathcal{I}$-class. Then for all
$a,b\in I$, IN$(a)$=IN$(b)$. Hence by Theorem \ref{i(a)=i(b)}, $I$
is a minimal interior ideal of $S$.

\end{proof}

\begin{theorem}\label{regular IJ}
Let $S$ be a semigroup. If $S$ is regular then
$\mathcal{J}=\mathcal{I}$.

\end{theorem}
\begin{proof}
Let $S$ be a regular semigroup. For any $a,b\in S$, $a\mathcal{I} b$
then IN$(a)$=IN$(b)$. Hence by Theorem \ref{interior-regular},
$I(a)=I(b)$ implies $a\mathcal{J} b$. Hence $\mathcal{I}\subseteq
\mathcal{J}$. By \ref{regular IJ}, $\mathcal{J}=\mathcal{I}$.
\end{proof}

\begin{corollary}
In an intra-regular semigroup $S$ $\mathcal{J}=\mathcal{I}$.
\end{corollary}

From Theorem \ref{I CLASS} and \ref{regular IJ}, we can conclude the
following result:

\begin{theorem}
Let $S$ be a regular semigroup and $I$ be an interior ideal of $S$.
$I$ is a minimal interior ideal of $S$ if and only if $I$ is a
$\mathcal{J}$-class.
\end{theorem}

\begin{center}
\bf{Acknowledgements}
\end{center}
The authors would like to thank the funding agency, the University
Grant Commission (UGC) of the Government of India, for providing
financial support for this research in the form of UGC-CSIR NET-JRF.

\bibliographystyle{amsplain}

\begin{thebibliography}{10}
\baselineskip 5mm








\bibitem{SZASZ1}
G. Szasz, Interior ideals in semigroups. In: Notes on semigroups IV,
Karl Marx Univ. Econ., Dept. Math. Budapest (1977), No. \textbf{5},
1-7.

\bibitem{SZASZ2}
G. Szasz, Remark on interior ideals of semigroups, Studia Scient.
Math. Hung. \textbf{16} (1981), 61-63.


\bibitem{Howie}
John. M. Howie, Fundumentals of semigroup theory, Claredon press,
ISBN 0-19-851194-9.

\bibitem{MOIN}

Moin A. Ansari, M. Rais Khan and J. P. Kaushik, A Note on (m,n)
Quasi-Ideals in Semigroups, Int. Journal of Math. Analysis,
Vol.\textbf{ 3}, 2009, no. \textbf{38}, 1853 - 1858.



\bibitem{Lajos2}
S. Lajos, On Characterization of Regular Semigroups, Proc. Japan
Acad., \textbf{44} (1968).

\bibitem{Lajos3}
S. Lajos, A New Characterization of Regular Duo Semigroups, Proc.
Japan Acad., \textbf{47} (1971).

\bibitem{LAJOS1}
S. Lajos, (m; k; n)-ideals in semigroups. In: Notes on Semigroups
II, Karl Marx Univ. Econ., Dept. Math. Budapest (1976), No.
\textbf{1}, 12-19.

\bibitem{S. Lajos 1}
S. Lajos, On the Bi-ideals in Semigroups, Proc. Japan Acad.,\textbf{
45} (1969).




\bibitem{S. Lajos}
S. Lajos, On the Bi-ideals in Semigroups. II,
 Proc. Japan Acad., \textbf{47 }(1971).





\end{thebibliography}

\end{document}